\documentclass[a4paper,11pt,oneside]{article}

\usepackage{amsmath,amssymb,amsthm,mathrsfs}
\usepackage{txfonts}
\usepackage{hyperref}

\usepackage[scr=boondoxo]{mathalfa}

\newcommand{\cB}{\mathscr{B}}
\newcommand{\cC}{\mathscr{C}}
\newcommand{\cD}{\mathscr{D}}
\newcommand{\cE}{\mathscr{E}}

\newtheorem{thm}{Theorem}
\newtheorem{lem}{Lemma}
\newtheorem{pro}{Proposition}

\newcommand{\R}{\mathbb{R}}
\newcommand{\Z}{\mathbb{Z}}

\newcommand{\diam}{\operatorname{diam}}
\newcommand{\asdim}{\operatorname{asdim}}
\newcommand{\Ndim}{\operatorname{dim}_{\text{\rm N}}}
\newcommand{\id}{\operatorname{id}}

\begin{document}

\title{Geodesic spaces of low Nagata dimension}
\author{Martina J{\o}rgensen \& Urs Lang}
\date{April 22, 2020}

\maketitle

\begin{abstract}
We show that every geodesic metric space admitting an injective continuous map into the plane as well as every planar graph has Nagata dimension at most two, hence asymptotic dimension at most two. This relies on and answers a question in a very recent work by Fujiwara and Papasoglu. We conclude that all three-dimensional Hadamard manifolds have Nagata dimension three. As a consequence, all such manifolds are absolute Lipschitz retracts.
\end{abstract}


\bigskip
In~\cite{FP}, Fujiwara and Papasoglu show that all planar geodesic metric spaces and planar graphs have asymptotic dimension at most three. Here, a geodesic metric space is said to be {\em planar} if it admits an injective continuous map into $\R^2$, and a not necessarily locally finite (but connected) graph, viewed as a geodesic metric space with edges of length one, is called {\em planar} if it admits an injective map into $\R^2$ whose restriction to every edge is continuous. In fact, Fujiwara and Papasoglu prove the stronger result that the (Assouad--)Nagata dimension, which is greater than or equal to the asymptotic dimension, is at most~3 (see below for the definitions of these notions). This is achieved by showing that for some universal constant $C$, every metric annulus of width comparable to $s > 0$ admits a covering by subsets of diameter at most $Cs$ such that every $s$-ball meets no more than two of them. For an appropriate sequence of annuli covering the underlying space $X$, the union of the individual covers has $s$-multiplicity at most $4$, which gives the bound on the dimension but exceeds the expected value by one.

In this note, we first observe that the bound on the Nagata dimension can be improved to $2$ (see Theorem~\ref{thm:planes}). This answers in particular Question~5.1 in~\cite{FP} for the asymptotic dimension. For the proof, it suffices to combine the above result for metric annuli with a simplified version of the Hurewicz-type theorem from~\cite{BDLM} for the case of a Lipschitz function $f \colon X \to \R$ (see Theorem~\ref{thm:new}). The latter holds in all dimensions, and we provide a streamlined proof. We then take a further step and apply Theorem~\ref{thm:planes} and Theorem~\ref{thm:new} to prove that all 3-dimensional Hadamard manifolds have Nagata dimension~3 (see Theorem~\ref{thm:hadamard}). In dimension two, this was shown in Theorem~5.7.3 in~\cite{S}, but in higher dimensions, previous results depend on additional curvature bounds or homogeneity assumptions (compare~\cite{G, LS}). Finiteness of the asymptotic or Nagata dimension has a number of important consequences (see~\cite{BD2} for a survey). It now follows from~\cite{LS} that all 3-dimensional Hadamard manifolds are absolute Lipschitz retracts (see Theorem~\ref{thm:alr}). Again, this subsumes a number of earlier results with extra conditions (compare~\cite{BuS, LPS}).


\medskip
We now state the relevant definitions. Let $(X,d)$ be a metric space. A collection $\cC$ of subsets of $X$ is called {\em $D$-bounded}, for some constant $D$, if every set $C \in \cC$ has diameter
\[
\diam(C) := \sup\{d(x,y) : x,y \in C\} \le D.
\]
Given $s > 0$, the {\em $s$-multiplicity} of $\cC$ is the infimum of all integers $m \ge 0$ such that every closed $s$-ball in $X$ meets at most~$m$ members of the collection. The {\em asymptotic dimension} $\asdim(X)$ of~$X$ is the infimum of all integers $n \ge 0$ for which there is a function $D \colon (0,\infty) \to (0,\infty)$ such that for every $s > 0$, $X$ possesses a $D(s)$-bounded covering of $s$-multiplicity at most $n+1$~\cite{G}.
The {\em Nagata dimension} or {\em Assouad--Nagata dimension} $\Ndim(X)$ of $X$ is defined analogously with a linear control function $D(s) = cs$~\cite{A, LS}. To make the constant $c > 0$ explicit, we say that $X$ has Nagata dimension $n$ or at most $n$ {\em with constant $c$}. It is an essential feature of the results in~\cite{FP} and those obtained here that they hold {\em uniformly} for all members of the class of spaces considered, that is, with the same constant $c$.

The following equivalent formulation is often useful. Let $\cC$ be a collection of subsets of $X$. For $s > 0$, $\cC$ is called {\em $s$-disjoint\/} if
\[
d(C,C') := \inf\{d(x,x') : x \in C,\,x'\in C'\} \ge s
\]
whenever $C,C' \in \cC$ are distinct. More generally, we say that $\cC$ is {\em $(n+1,s)$-disjoint\/} if $\cC = \bigcup_{i=1}^{n+1} \cC_i$ for subcollections $\cC_i$ that are individually $s$-disjoint. We think of the indices $1,\ldots,n+1$ as {\em colours} of $\cC$. Evidently, if $\cC$ is $(n+1,s)$-disjoint, then $\cC$ has $\lambda s$-multiplicity at most $n+1$ for any $\lambda \in (0,\frac12)$. Hence, a metric space $X$ with an $(n+1,s)$-disjoint and $c's$-bounded cover for every $s > 0$ satisfies $\Ndim(X) \le n$ with any constant $c > 2c'$. Conversely, the following holds (see the proof of Proposition~2.5 in~\cite{LS}).

\begin{pro} \label{pro:colours}
If\/ $X$ is a metric space with a $cs$-bounded cover of $s$-multiplicity at most $n+1$ for some $c,s > 0$, then $X$ also admits an $(n+1,\lambda s)$-disjoint and $c'\lambda s$-bounded cover, where $\lambda,c' > 0$ depend only on $c$ and $n$. In particular, if\/ $X$ has Nagata dimension at most $n$ with constant $c$, then $X$ possesses an $(n+1,s)$-disjoint and $c's$-bounded cover for every $s > 0$. 
\end{pro}


We now proceed to the aforementioned version of the Hurewicz theorem. We need the following lemma extracted from Theorem~2.4 in~\cite{BDLM}.

\begin{lem}\label{lem:more}
Let $X$ be a metric space with an $(n+1,s)$-disjoint and $D$-bounded cover for some $n \ge 1$ and $s,D > 0$. Then there exists an $(n+2,\tfrac{s}{3})$-disjoint and $(D+\tfrac{2}{3}s)$-bounded cover of\/ $X$ with each point $x \in X$ belonging to at least two sets of different colours.
\end{lem}	

\begin{proof}
By assumption, $X$ has an $(n+1,s)$-disjoint and $D$-bounded cover $\cC = \bigcup_{k=1}^{n+1}\cC_k$. For $C \in \cC$, let $C'$ denote the closed $\frac{s}{3}$-neighbourhood of $C$. Define a new cover $\bigcup_{k=1}^{n+1}\cC_k'$, where $\cC_k' := \{C': C \in \cC_k\}$. Notice that this cover is $(n+1,\tfrac{s}{3})$-disjoint and $(D+\tfrac{2}{3}s)$-bounded. Now define sets of an additional colour as follows. For every $B \in \cC$ of colour $j$, let $B^{\,0}$ be the set $B$ minus the union of all $C' \in \bigcup_{k \ne j}\cC_k'$, and put $\cC_{n+2}' := \{B^{\,0} : B \in \cC\}$. The cover $\cC' := \bigcup_{k=1}^{n+2}\cC_k'$ will have the required
properties.

Clearly, $\cC'$ is still $(D + \frac{2}{3}s)$-bounded, as $\diam(B^{\,0}) \le \diam(B) \le D$ for all $B \in \cC$. To verify that $\cC'$ is $(n+2,\frac{s}{3})$-disjoint, it remains to check that $d(A^0,B^{\,0}) \ge \frac{s}{3}$ whenever $A \in \cC_i$, $B \in \cC_j$, and $A \ne B$. If $i = j$, then $d(A^0,B^{\,0}) \ge d(A,B) \ge s$, because $\cC_i$ is $s$-disjoint. If $i \ne j$, then $A^0 \subset A$ and $B^{\,0} \subset B \setminus A'$ by construction of $\cC_{n+2}'$, thus $d(A^0,B^{\,0}) \ge d(A, B \setminus A') \ge \tfrac{s}{3}$. Now let $x\in X$. There exist $j \in \{1,\ldots,n+1\}$ and $B \in \cC_j$ such that $x \in B$. If $x$ belongs to some set $C' \in \cC_k'$ with $k \ne j$, then $x \in B' \cap C'$. If there is no such set $C'$ containing $x$, then $x \in B' \cap B^{\,0}$. Thus in either case, $x$ belongs to at least two sets of different colours.
\end{proof}

The following result corresponds to Theorem~7.2 in~\cite{BDLM} for the case of a Lipschitz function $f \colon X \to \R$. The argument can also easily be adapted to the asymptotic dimension (see Theorem~1 in~\cite{BD1} for an earlier result).

\begin{thm} \label{thm:new}
Let $X$ be a metric space, and let $f \colon X \to \R$ be a $1$-Lipschitz function. Suppose that there exist $n \ge 1$ and $c > 0$ such that for all $r \in \R$ and $s > 0$, the set $f^{-1}([r,r+s))$ possesses an $(n+1,s)$-disjoint and $cs$-bounded cover. Then $X$ has Nagata dimension at most $n+1$ with a constant depending only on $n$ and $c$.
\end{thm}

\begin{proof}
Fix $s > 0$ and let $t := (n+2)s$. For every $k \in \Z$, put $I_k := [kt,(k+1)t)$.

For every odd integer $k$, proceed with the following construction. By assumption, $f^{-1}(I_k)$ admits an $(n+1,t)$-disjoint and $ct$-bounded cover. Lemma~\ref{lem:more} now provides an $(n+2,\tfrac{t}{3})$-disjoint and $(c+1)t$-bounded cover $\cC_k = \bigcup_{i=1}^{n+2}\cC_k^i$ with each point $x \in f^{-1}(I_k)$ belonging to two sets of different colours. Let $\cB_k = \bigcup_{i=1}^{n+2}\cB_k^i$ be the $(n+2,s)$-disjoint cover of $I_k$ such that $\cB_k^i$ consists of the connected components of $I_k \setminus [kt + (i-1)s,kt + is)$. Every point of $I_k$ is in $n+1$ sets of pairwise different colours. For $i = 1,\ldots,n+2$, define
\[
\cD_k^i := \bigl\{ C \cap f^{-1}(B) : C \in \cC_k^i,\,B \in \cB_k^i \bigr\}.
\]
Note that if $x \in f^{-1}(I_k)$, then $x$ belongs to two sets of different colours of $\cC_k$, and $f(x)$ is in $n+1$ sets of pairwise different colours of $\cB_k$; thus there exists an index $i \in \{1,\ldots,n+2\}$ such that $x$ belongs to some set $C \cap f^{-1}(B) \in \cD_k^i$. Hence, $\bigcup_{i=1}^{n+2}\cD_k^i$ is a covering of $f^{-1}(I_k)$. Since $\cC_k$ is $(c+1)t$-bounded, this cover is $c's$-bounded for $c' := (n+2)(c+1)$, and since $s \le \frac{t}{3}$ and $f$ is $1$-Lipschitz, $\cD_k^i$ is $s$-disjoint. For $i \in \{1,\ldots,n+2\}$, let $\cD^i := \bigcup_{\text{$k$ odd}}\cD_k^i$. Note that $\cD^i$ is still $s$-disjoint.

Now put $s' := (c'+2)s$ and note that $s' \ge (n+2)s = t$. For every even integer $k$, there exists by assumption an $(n+1,s')$-disjoint and $cs'$-bounded cover $\bigcup_{i=1}^{n+1}\cE_k^i$ of $f^{-1}(I_k)$ by subsets of $f^{-1}(I_k)$. For $i \in \{1,\ldots,n+1\}$, let $\cE^i := \bigcup_{\text{$k$ even}} \cE_k^i$. The union
\[  
\bigcup_{i=1}^{n+1}\,\bigl(\cD^i \cup \cE^i\bigr) \cup \cD^{n+2}
\]
is an $(n+2)$-coloured covering of $X$ which we shall modify to satisfy the required properties. For every $E \in \cE^i$, let $E^*$ be the union of $E$ with all sets of $\cD^i$ at distance $< s$ from $E$. Clearly $\diam(E^*) \le cs' + 2(c'+1)s = c''s$ for some $c'' = c''(n,c)$. We claim that $d(E^*,F^*) \ge s$ whenever $E,F \in \cE^i$ are distinct. If $E$ and $F$ belong to the same family $\cE_k^i$, this holds since $d(E,F) \ge s' \ge 2s + \diam(D)$ for all $D \in \cD^i$ and $\cD^i$ is $s$-disjoint. In the other case, $d(E,F) \ge t \ge s$, and by construction no set $D \in \cD^i$ is at distance $< s$ from both $E$ and $F$, so the claim follows again since $\cD^i$ is $s$-disjoint. In the final covering, the collection of sets of colour $i$ consists of all $E^*$ with $E \in \cE^i$ and the remaining elements of $\cD^i$ not belonging to such an $E^*$. This gives an $(n+2,s)$-disjoint and $c''s$-bounded cover of $X$. Since $s > 0$ was arbitrary, $\Ndim(X) \le n+1$.
\end{proof}


We now turn to planar geodesic spaces. The proof of Theorem~\ref{thm:planes}
below relies on the following result from~\cite{FP}. 

\begin{pro} \label{pro:annuli}
There is a universal constant $c_1$ such that the following holds. Suppose that $X$ is a planar geodesic metric space or a planar graph, and let $z \in X$ be a base point. Then for any $r,t > 0$, the metric annulus $\{x \in X : r \le d(z,x) \le r+t\}$ admits a $c_1t$-bounded cover of $t$-multiplicity at most~2.
\end{pro}

See Lemma~4.4 in~\cite{FP} for $M = 10m$. One can take $c_1 = 10^6$.

\begin{thm} \label{thm:planes}
There is a universal constant $c_2$ such that every planar geodesic metric space or planar graph has Nagata dimension at most $2$ with constant $c_2$.  
\end{thm}

\begin{proof}
Let $X$ be a planar geodesic metric space or a planar graph, let $z \in X$ be a base point, and put $f := d(z,\cdot\,)$. By Proposition~\ref{pro:annuli}, for any $r,t > 0$, the set $f^{-1}([r,r+t])$ admits a $c_1t$-bounded cover of $t$-multiplicity at most~2. It follows from Proposition~\ref{pro:colours} that $f^{-1}([r,r+t])$ possesses a $(2,\lambda t)$-disjoint and $c_1'\lambda t$-bounded cover, where $\lambda$ and $c_1'$ depend only on $c_1$. If $\lambda < 1$, put $s := \lambda t$ and $c := c_1'$. If $\lambda \ge 1$, put $s := t$ and $c := c_1'\lambda$. In either case, for any $r,s > 0$, the set $f^{-1}([r,r+s])$ has a $(2,s)$-disjoint and $cs$-bounded cover. Now the result follows from Theorem~\ref{thm:new}.
\end{proof}  


We proceed to complete, simply connected Riemannian manifolds of nonpositive sectional curvature.

\begin{thm} \label{thm:hadamard} 
There is a universal constant $c_3$ such that every 3-dimensional Hadamard manifold has Nagata dimension $3$ with constant $c_3$.    
\end{thm}

\begin{proof}
Let $(X,d)$ be a 3-dimensional Hadamard manifold. Fix a Busemann function $f \colon X \to \R$. Recall that $f$ is $1$-Lipschitz, all horoballs $f^{-1}((-\infty,r])$ are convex, and the nearest point retraction from $X$ onto any horoball is $1$-Lipschitz. Furthermore, every horosphere $f^{-1}\{r\}$ is homeomorphic to $\R^2$, and the induced inner metric on $f^{-1}\{r\}$ is finite and complete, hence geodesic. (In fact, Busemann functions and horospheres are $C^2$, see Proposition~3.1 in~\cite{HI}, but this is not needed here.) By Theorem~\ref{thm:planes}, every $f^{-1}(r)$ has Nagata dimension at most~2 with constant $c_2$. Now fix $r \in \R$ and $s > 0$. By Proposition~\ref{pro:colours} there is a universal constant $c_2'$ such that the horosphere $H := f^{-1}\{r-\frac{s}{2}\}$, with the induced inner metric $d_H$, possesses a $(3,s)$-disjoint and $c_2's$-bounded cover $\cB = \bigcup_{i=1}^3\cB_i$ by subsets of $H$. Equip $A := f^{-1}([r-\frac{s}{2},r+s])$ with the induced inner metric $d_A$, and let $\pi \colon A \to H$ be the nearest point projection. Note that $\pi$ does not increase the length of curves, thus $\pi$ is $1$-Lipschitz also with respect to $d_A$ and $d_H$. The sets $\pi^{-1}(B)$ with $B \in \cB$ form a $(3,s)$-disjoint and $(c_2' + 3)s$-bounded cover of $(A,d_A)$. Furthermore, by taking the intersections with $A' := f^{-1}([r,r+s])$, we get a cover of $A'$ that is $(3,s)$-disjoint and $(c_2' + 3)s$-bounded with respect to $d$, because $d(x,y) \le d_A(x,y)$ for all $x,y \in A'$, with equality if $d(x,y) \le s$. Since $r \in \R$ and $s > 0$ were arbitrary, it follows from Theorem~\ref{thm:new} that $X$ has Nagata dimension at most 3 with a universal constant~$c_3$. In fact, $\Ndim(X) = 3$, as $X$ is locally bi-Lipschitz homeomorphic to $\R^3$.
\end{proof}

A metric space $X$ is an {\em absolute $C$-Lipschitz retract}, for a constant $C \ge 1$, if for every isometric embedding $e$ of $X$ into another metric space $Y$ there is a $C$-Lipschitz retraction of $Y$ onto $e(X)$. Equivalently, for every metric space $W$ and every $\lambda$-Lipschitz map $f \colon V \to X$ with $V \subset W$, there exists a $C\lambda$-Lipschitz extension $\bar f \colon W \to X$ of $f$ (see Proposition~1.2 in~\cite{BeL}). It was shown in~\cite{LPS} that Hadamard manifolds with pinched negative sectional curvature, homogeneous Hadamard manifolds, and all 2-dimensional Hadamard manifolds are absolute Lipschitz retracts (in the last case, one can take $C = 4\!\sqrt{2}$). Furthermore, according to~\cite{BuS}, the universal cover of every closed Riemannian 3-manifold of nonpositive curvature has this property. Combining Theorem~\ref{thm:hadamard} with Theorem~1.6 from~\cite{LS} (the case $f = \id$ on $Y = Z$), we can now settle the 3-dimensional case completely.

\begin{thm} \label{thm:alr}
There is a universal constant $c_4$ such that every 3-dimensional Hadamard manifold is an absolute $c_4$-Lipschitz retract.   
\end{thm}  


\bigskip\noindent
\texttt{martinaj@student.ethz.ch\\
lang@math.ethz.ch}\medskip\\
Department of Mathematics\\
ETH Z\"urich\\
R\"amistrasse 101\\
8092 Z\"urich\\
Switzerland

\end{document}